\documentclass[12pt]{article}

\usepackage{amsthm,amsmath,amsfonts,epsfig}
\usepackage{natbib}
\usepackage{epstopdf}
\setcitestyle{authoryear, open={(},close={)}}

\theoremstyle{plain}
\newtheorem{thm}{Theorem}[section]
\newtheorem*{thm*}{Theorem}

\newtheorem*{prop*}{Proposition}

\newtheorem*{Lemma*}{Lemma}

\theoremstyle{remark}
\newtheorem{rem}[thm]{Remark}
\newtheorem{defin}{Definition}
\newtheorem{example}{Example}

\newcommand{\norm}[2]{\left\Vert #2\right\Vert_{#1}}
\newcommand{\gep}{\varepsilon}

\newcommand{\qqd}{\qquad}

\newcommand{\figuraqui}[3]{\begin{figure}[h!]\begin{center}\includegraphics[width=#3in]{#2}\caption{#1}\end{center}\end{figure}}

\newcommand{\en}{\in}
\newcommand{\abs}[1]{\left\vert\, #1\,\right\vert}
\newcommand{\espe}[2]{\mathbb{E}_{#1}\left[#2\right]}
\newcommand{\proba}[2]{\mathbb{P}_{#1}\left[#2\right]}
\newcommand{\ops}{\mathcal{B}(H)}
\newcommand{\prodint}[2]{\left<#1,#2\right>}
\newcommand{\indica}[1]{\mathversion{bold} 1\!\!1\left(#1\right)}

\providecommand{\keywords}[1]{\textbf{Keywords:} #1}

\title{The Functional AR(1) process with a unit root}
\author{Nelson Muriel\\ \small{Centro de Investigaci\'on en Matem\'aticas A.C.}\\ \small{Jalisco S/N, Col. Valenciana}\\
\small{36240, Guanajuato, Mexico}}
\date{}
\begin{document}

\maketitle

\begin{abstract}
We define strong and weak unit roots for the functional AR(1) process and
give some theoretical examples. It is shown that a functional form of cointegration
occurs in which only a finite number of common trends exist.  Using functional
Principal Component Analysis we illustrate the presence of functional unit roots
in two demographic data sets.  We close with some remarks concerning our assumptions
and the possibility of generalizing our results.
\end{abstract}


\section{Introduction}

Random variables with values in a functional space arise naturally in many fields. Examples of its use in areas as diverse as Criminology, Paleopathology or Medicine, among others, are provided in \cite{ramsay02}.  The theory of inference for stochastic processes in continuous time of \cite{Grenander81} is built on the basis of viewing such processes as infinite dimensional random variables.  Via this identification, prediction of a continuous time stochastic processes $\{Y_t\}$ becomes viable by defining
\[\{X_n = Y_t\indica{t\en [n\delta, (n+1)\delta]}, n\en \mathbb{Z}\}.\]
The choice of $\delta$ will depend of the particular application at hand.  Observe that the temporal index of the process has switched from continuous to discrete, so that $\{X_n\}$ defines an infinite dimensional time series.

Just as in the finite-dimensional case, linear time series in functional spaces provide a good approximation for stationary processes.  Also, some well known diffusions can be represented linearly in functional spaces as shown in \cite{RePEc:eee:stapro:v:76:y:2006:i:1:p:93-108} or \cite{bosq00} for the Ornstein-Uhlenbeck process.  A most successful approach to linear modeling in functional spaces is the infinite-dimensional analog of traditional AR(1) processes.

Let $(H,\prodint{}{})$ denote a real, separable Hilbert space and let $\ops$ be the algebra of all the operators acting from $H$ to $H$.  For a given $\rho\en\ops$, define the stochastic process $\{X_n\}$ as the solution to the equations

\begin{equation}\label{e:far1}
X_n = \rho(X_{n-1})+\gep_n,
\end{equation}
with $\{\gep_n\}$ a $H$--white noise sequence.  This process admits an obvious generalization to the AR(p) specification which is developed in \cite{bosq00}.

From the point of view of applications, the selection of the order for a functional AR process is discussed in \cite{Kokoszka12}.  As the authors point out, only small values of $p \en \{0,1,2\}$ are worth considering due to the complexity of the marginal processes induced by \eqref{e:far1}.  This makes the richness of the functional AR(1) dynamics clear and justifies our decision to focus solely on this specification.

One of the major assumptions made when dealing with the functional AR(1) process in applications is that it is stationary.  A known condition for stationarity is given in \cite{bosq00}, and asks that $\norm{}{\rho^j}<1$ for some $j\en\mathbb{N}$.  Should this condition be met, the process $\{X_n\}$ admits the representation
\[X_n = \sum_{j=0}^\infty \rho^j(\gep_{n-j}).\]
It is thus clear that given any fixed $v\en H$, the real process $\{\prodint{X_n}{v}\}$ is also stationary with representation
\[\prodint{X_n}{v}=\sum_{j=0}^\infty \prodint{\gep_{n-j}}{\rho^{*j}(v)},\]
where $\rho^*$ is the adjoint of $\rho$.

As a consequence, the expansion of $\{X_n\}$ in whatever basis for $H$ will necessarily produce stationary coefficients.  This also means that the coefficients of the functional principal components (see \cite{ramsay05}) of the observed data should be stationary.  However, in some applications this is not the case.  The reason for this departure from stationarity in the estimated Principal Components (PC) coefficients may be twofold.  First, it is possible that the AR(1) specification is not stable throughout the sampling time.  Second, it may be that the AR(1) specification is stable but not stationary.  The first case has been studied in \cite{Horvath10} and the second is the topic of this paper.

It is often the case when dealing with the functional AR(1), that $\rho$ in \eqref{e:far1} is assumed integral, Hilbert-Schmidt, compact or diagonalizable. See, for example \cite{Kokoszka12}, \cite{RePEc:eee:stapro:v:76:y:2006:i:1:p:93-108}, \cite{RuizMedina2007} or \cite{raey}.  These assumptions are not really limitations to the model.  For instance, it is well known that the ideal of compact operators is the norm closure of the space of finite-rank operators as shown in \cite{ConwayACFA}.  As the statistical analysis of functional data proceeds by finite rank projections, the assumption that $\rho$ is compact amounts to saying that it can be properly approximated by in-sample operations. In particular, compactness is necessary for consistency.  We will throughout this paper assume compactness.

The organization of the paper is as follows. Functional unit roots are defined in Section 2 in terms of the point spectrum of $\rho$.  A strong and weak form are considered and we provide some basic examples.  Section 3 explores the main structural consequences of a functional unit root which admit an interpretation analogous to its finite-dimensional counterpart.  Particularly, a suitable definition of (linear) cointegration is given and a cointegrating space is shown to exist in Theorem \ref{t:unitroots}.  Furthermore, the space of common trends is shown to be finite-dimensional.  Section 4 explores two data sets from demography to illustrate the presence of functional unit roots.  The first one consists of observations of the male log-mortality rate in Italy and the second one of Australian fertility rates.  Section 5 concludes our exposition providing some remarks on our assumptions and the possibility of generalizing our results.

\section{Definitions and Examples}

Given $\rho\en\ops$ we denote by $\norm{}{\rho}$ its operator norm, by
$\sigma(\rho$) its spectrum and by $r(\rho)$ its spectral radius. Letting $I$ denote the identity operator in $\ops$, we use the isomorphism $\lambda\mapsto\lambda I$ and thus assume that $\mathbb{R}\subseteq \ops$. The use of $\lambda\en\mathbb{R}$ either as a real number of as an operator should rise no confusion.

It is proven in \cite{bosq00} that the stochastic process \eqref{e:far1} admits a stationary solution if for some $j\en\mathbb{N}$, we have $\norm{}{\rho^j}<1$.  Due to the algebraic structure of $\ops$ it is immediate that this condition is equivalent with $r(\rho)<1$.

As pointed out in the Introduction, we will assume that $\rho$ is a compact operator.  Thus the spectrum of $\rho$ is a discrete set $\{\lambda_n\}$ with $0$ as its only possible accumulation point.  The condition for stationarity becomes
\begin{equation*}
\abs{\lambda_n}<1,\qqd \forall n\en\mathbb{N}.
\end{equation*}

Since the process $\{X_t\}$ takes it values in a Hilbert space $H$, two kinds of non-stationarity may arise.  First, a strong form in which $\{X_t\}$ itself is non stationary as an $H$-valued random process. Second, a weak form in which for some $v\en H$, $\{\prodint{X_t}{v}\}$ is non-stationary as a real-valued process.

The stationarity of $\{X_t\}$ entails that of $\{\prodint{X_t}{v}\}$ for all $v\en H$, so that weak non--stationarity implies strong non--stationarity.  The following definitions make it clear which forms of non--stationarity we will be interested in.

\begin{defin}
The process \eqref{e:far1} is said to have a strong unit root if $1\en\sigma(\rho)$ and $\abs{\lambda_n}<1$ for all other $\lambda_n\en \sigma(\rho)$.
\end{defin}

For weak unit roots, we will use the notion of integrated processes of order one which will be denoted by $I(1)$.  A comprehensive account on such processes can be found in \cite{Johansen95}.

\begin{defin}
Given $v\en H$, we call $v$ a weak unit root for $\{X_t\}$ if the real-valued process $\{\prodint{X_t}{v}\}$ is $I(1)$.  The set $\{v\en H, v \text{ is a weak unit root of } \{X_t\}\}$ will be denoted by $W_X$.
\end{defin}

\begin{example}\label{ex:K}
Let $H=L^2([0,1])$ with the usual Lebesgue measure. Let $K$ be a separable Kernel of the form
\[K(s,t) = \sum_{i=1}^n a_i(s)b_i(t),\]
and define
\[\rho(v)(s) = \int_0^1 K(s,t)v(t)dt.\]
Then,
\begin{equation*}
\begin{split}
\rho (v) (s) &= \int_0^1 \sum_{i=1}^n a_i(s)b_i(t)v(t)dt\\
             &=\sum_{i=1}^n \left(\int_0^1 b_i(t)v(t)dt\right)a_i(s).
\end{split}
\end{equation*}

The operator $\rho$ is thus finite-rank and its range is the span of $\{a_i, i=1,\dotsc, n\}$.    The equation $\lambda v = \rho(v)$ is therefore only meaningful for function $v=\sum_{j=1}^n v_j a_j$ and in this case we have
\begin{equation*}
\begin{split}
\lambda v(s) &= \lambda \sum_{i=1}^n v_i a_i(s) = \sum_{i=1}^n a_i(s)\int_0^1 b_i(s)\sum_{j=1}^n v_j a_j(s)ds\\
             &=\sum_{i=1}^n a_i(s)\sum_{j=1}^n\gamma_{i,j}v_j,
\end{split}
\end{equation*}
where
\[\gamma_{i,j}=\int_0^1 a_j(t)b_i(t)dt.\]
Let $V=(v_1,\dotsc, v_n)^T$ and $A_K = \left(\gamma_{i,j}\right)_{i,j}$. From the previous calculations it is not hard to see that the eigenvalue-eigenvector equation for $\rho$ is equivalent to the finite-dimensional equation $\lambda V = A_K V$.

Thus, the functional process $\{X_n\}$ has a strong unit root if and only if the finite-dimensional process
\[Y_n = A_K Y_{n-1}+\eta_t\]
is integrated of order 1.  It would be, for instance, sufficient that $1$ is the only root of the characteristic polynomial $\det(I-A_Kz)$ with unit modulus.
\end{example}

\begin{example}[From Example 3.7 in \cite{bosq00}]\label{ex:B}
Let $\alpha\en[-1,1]$ and define
\[\rho(v) = \alpha \left(\prodint{v}{e_1}+\prodint{v}{e_2}\right)e_1 + \alpha \prodint{v}{e_1}e_2,\]
where $\{e_1, e_2\}$ is an orthonormal system in $H$. Assume that the white noise $\{\gep_n\}$ satisfies $\espe{}{\prodint{\gep_n}{e_1}^2}>0$ but $\espe{}{\prodint{\gep_n}{e_2}^2}=0$.

Since $\{\gep_n\}$ has mean zero, this last condition implies that $\prodint{\gep_n}{e_2}=0\; a.s.$ for all $n$.  Therefore, letting $\gep$ denote a generic noise variable, a recursive application of $\rho$ shows that
\[\rho^j (\gep) = \alpha^j\prodint{\gep}{e_1}(f(j)e_1+f(j-1)e_2),\]
where $\{f(j)\}$ stands for the Fibonacci sequence.  It follows that the process $\{X_n\}$ is stationary for values of $\alpha$ such that $\alpha+\alpha^2<1$.

In fact, following \cite{bosq00}, it can be seen that the process $\{Y_n=\prodint{X_n}{e_1}\}$ satisfies the equation
\[Y_n = \alpha Y_{n-1}+\alpha^2 Y_{n-2}+\prodint{\gep_n}{e_1},\]
so that for $\alpha + \alpha^2 = 1$, the real process $\{Y_n\}$ is $I(1)$.  Therefore the functional process $\{X_n\}$ has the weak unit root $e_1$ for such value of $\alpha$.

Since $\rho$ is a finite rank operator, calculations similar to the previous example will show that the eigenvalue-eigenvector equation for $\rho$ is equivalent to
\[\lambda x = \left(\begin{array}{cc}\alpha&\alpha\\ \alpha&0\end{array}\right)x.\]
For the value $\alpha=-\tfrac{1}{2}+\tfrac{\sqrt{5}}{2}$, which makes $e_1$ a weak unit root, we find that $1$ is an eigenvalue of $\rho$ and the other eigenvalue has modulus less than one, so that $\{X_n\}$ has a strong unit root.

This example illustrates that weak unit roots do not necessarily arise from a simple random walk representation
\[\prodint{X_n}{v}=\prodint{X_{n-1}}{v}+\prodint{\gep_n}{v},\]
but actually from the more general condition $\{\prodint{X_n}{v}\}\en I(1)$.  A generalization to an AR(p) specification for $\{\prodint{X_t}{e_1}\}$ can be built such that
\[\prodint{X_t}{e_1}=\sum_{i=1}^p \alpha^i \prodint{X_{t-i}}{e_1}+\prodint{\gep_n}{e_1}.\]
The value $\alpha$ for which $\sum_{i=1}^p\alpha^i=1$ will make a weak unit root of $e_1$ .
\end{example}

\begin{example}
Assume that $\rho$ is a trace-class operator, so that $\rho$ is compact and
\[\sum_n\abs{\lambda_n}<\infty.\]
The Fredholm determinant of $-\rho$ is defined as the entire function
\[p(z) = \det(I-z\rho) =\prod_{n}(1-\lambda_n z),\quad z\en\mathbb{C}.\]
The absolute summability of $\{\lambda_n\}$ implies that $p(z)$ takes on the value zero if and only if one of its terms is zero.  Thus, the roots of $p(z)$, counting multiplicities, are exactly
\[z_n = \frac{1}{\lambda_n}.\]
Therefore the process \eqref{e:far1} will have a unit root if and only if $1$ is a root of $p(z)$ and all other roots are outside the complex unit disc.  The function $p(z)$ is the exact analog of the characteristic polynomials for trace-class operators.

In particular, if $\rho$ is an integral operator with $L^2$ Kernel it is trace-class.  The numerical evaluation of $p(z)$ has been discussed in this case, and a method for its computation can be found in \cite{bornemann10}.

\end{example}

\begin{example}\label{ex:OU}
Let $H=L^2([0,1])$ and define $\rho(v) = e^{-\theta t}v(1)$.  If $\theta>0$, this operator has been used to express the Ornstein--Uhlenbeck process as a functional AR in \cite{bosq00} and \cite{RePEc:eee:stapro:v:76:y:2006:i:1:p:93-108}.  Since the range of $\rho$ is the linear span of the function $e^{-\theta t}$ the eigenvalue-eigenvector equation for $\rho$ is only meaningful for $x=\alpha e^{-\theta t}$ in which case we have
\[e^{-\theta t} \alpha e^{-\theta} = \lambda \alpha e^{-\theta t}.\]
It follows that the only eigenvalue is $\lambda = e^{-\theta}$.  Therefore, the process only has a unit root when $\theta=0$ in which case the process being represented is actually a Brownian Motion.
\end{example}

\section{The structure of a functional unit root}

We begin this section with some observations from the finite-dimensional case.  An integrated $p$-dimensional VAR(1) process can be described through
\[\triangle X_n = \Pi X_{n-1}+\eta_n,\]
where the matrix $\Pi$ is of rank $r<p$.  This condition leads to the well-known Granger's representation according to which
\begin{equation}\label{e:coint}
X_n = C\sum_{j=1}^n \eta_j + C(L) \eta_n.
\end{equation}
The matrix $C$ appearing in \eqref{e:coint} is of rank $r$ and the function $C(L)$ represents a linear filter with summable coefficients.  The reader is referred to \cite{Johansen95}, \cite{RePEc:aea:aecrev:v:81:y:1991:i:4:p:819-40}, or \cite{StockWatson88} among many others.  One of the consequences of such a representation is that in the space orthogonal to the range of $C$, the process $\{X_n\}$ is stationary, while it has a random walk component in this range space.   More precisely, if $\beta\en \mathbb{R}^p$ is such that $\beta^T C= 0$ then the real process $\{\beta^T X_n\}$ is stationary. Each such $\beta$ is known as a cointegrating relation among the variables of $\{X_n\}$, and is often interpreted as a form of equilibrium in the dynamics of the process.  The linear span of these vectors is known as the cointegrating space and is characterized by the property that $\{X_n\}$ is stationary when projected onto it.

It is worth mentioning that those vectors $\{\beta\}$ are not uniquely determined, but that the cointegrating space is.  Therefore, cointegration can be explained as a partition of $\mathbb{R}^p$ into two orthogonal (complemented) closed subspaces.  The first one of them, namely, the orthogonal complement of $Ker(C)$ in \eqref{e:coint}, is responsible for most of the variability observed in any sample path (since in it, $\{X_n\}$ behaves like a random walk).  The second space, $Ker(C)$, implies only minor variability in the sample paths.  This feature of cointegration has made the techniques from multivariate statistics useful for estimating the cointegrating space in finite dimensions as in \cite{muriel12} or \cite{Snell99}.

Taking this point of view of cointegration is particularly useful for its extension into infinite dimensions.

\begin{defin}
The functional process \eqref{e:far1} is said to cointegrate if there exists an operator $T\en \ops$ with closed range such that $\{T(X_n)\}$ is a stationary, linear functional process.
\end{defin}

The closed range assumption goes in the spirit of the previous discussion.  Since in Hilbert spaces all closed subspaces are complemented, this definition states that $H= C_X \oplus U_X$ and $\{X_n\}$ is stationary on $C_X$.  The following Theorem shows that the functional AR(1) process with unit roots cointegrates. For the proof, we will show that $\rho$ induces an appropriate decomposition of $H$.  The projection of $\{X_n\}$ to $U_X$ is a pure random walk, while that to $C_X$ is a stationary AR(1) process.

\begin{thm}\label{t:unitroots}
Let $\{X_n\}$ be the AR(1) process \eqref{e:far1} with $\rho$ compact, and assume that $\{X_n\}$ has a strong unit root.  Then
\begin{enumerate}
\item There exists a projection operator $\Pi_U$, which commutes with $\rho$, onto a closed subspace $U_X\le H$ such that $\{\Pi_U(X_n)\}$ is a pure random walk and $\dim(U_X)<\infty$
\item There exists a projection operator $\Pi_S$, which commutes with $\rho$, onto a subspace $C_X\le H$ such that $\{\Pi_S(X_t)\}$ is a stationary process functional AR process.
\item The spaces are complementary in the sense that $H = U_X \oplus C_X$
\end{enumerate}
\end{thm}

\begin{proof}
Since $\rho$ is assumed to be compact, $\dim(Ker(\rho-\lambda))<\infty$ for every $\lambda\neq 0 \en \sigma(\rho)$.  Since the process $\{X_t\}$ has a unit root, it follows that $1\en\sigma(\rho)$.  Thus, let $U_X = Ker(\rho-I)$.  Being a closed finite--dimensional subspace, it is complemented in $H$.  Let $C_X$ be its orthogonal complement so that 3. in our Theorem is satisfied.\par
Let $\Pi_U$ be the finite rank projection onto $U_X$ and $\Pi_S= I-\Pi_U$.  Then $\sigma(\rho\vert_{\Pi_U(H)})=\{1\}$ and $\sigma(\rho\vert_{\Pi_S(H)})=\sigma(\rho)\setminus\{1\}$.  Since these projection operators commute with $\rho$ it follows that
\[\Pi_U(X_n) = \Pi_U\rho(X_{n-1})+\Phi_U(\gep_n) = \rho\Pi_U (X_{n-1})+\Pi_u(\gep_n).\]
The restriction of the spectrum now implies that
\[\Pi_U(X_n)=\Pi_U(X_{n-1})+\Pi_U(\gep_n).\]
Since $\Pi_U(\gep_n)$ is white noise in $U_X$, it follows that $\Pi_U(X_n)$ is a pure random walk proving 1. Since the spectral radius of $\rho\vert_{\Pi_S(H)}$ is less than unity, a similar reasoning shows that $\Pi_S(X_n) = \rho\left(\Pi_S(X_{n-1})\right)+\Pi_S(\gep_n)$ defines a stationary AR(1) process, concluding the proof.
\end{proof}

We now provide some remarks on the Theorem which explain a bit further the implications of a functional unit root.

\begin{rem}
It follows from the Theorem that $X_t = \Pi_U X_t + \Pi_S X_t$.  Since projection operators are self--adjoint and leave its defining space unaltered, if $v\en U_X\setminus\{0\}$ we have
\[\prodint{X_n}{v}=\prodint{X_n}{\Pi_Uv}=\prodint{\Pi_U X_n}{v}.\]
Therefore $U_X\setminus\{0\}\subseteq W_X$.  By a similar reasoning, if $x\en C_X$ then $x\not\in W_X$ so that $W_X\subseteq C_X^c$.
\end{rem}

\begin{rem}
Since $\dim(U_X)=m<\infty$, let $\{\phi_1,\dotsc, \phi_m\}$ be an orthonormal basis for it.   Writing $v=\sum_{i=1}^m \alpha_i \phi_i$, it becomes apparent that for any given $v\en U_X$,
\[\prodint{X_n}{v}=\sum_{i=1}^m \alpha_i\prodint{X_n}{\phi_i}.\]
In this sense, the weak unit roots generated by elements in $U_X$ can be represented by the $m$--dimensional process $\{\prodint{X_n}{\phi_i}, i=1,\dotsc, m\}$.

Observe that this process does not cointegrate, while due to the representation just mentioned, $k$-dimensional processes of the form
\[\{\prodint{X_n}{v_i}, i=1,2,\dotsc, k; v_i\en U_X\}\]
may indeed cointegrate having some of the elements in $\{\prodint{X_t}{\phi_i}, i=1,2,\dotsc, m\}$ as common trends.
\end{rem}

\begin{rem}
The condition that $v\en W_X$ is easily seen to be equivalent to
\[\proba{}{\prodint{X_{n-1}}{\rho^*(v)-v}=0}=1, \quad \forall n.\]
Therefore,
\[W_X = Ker(\rho^*-I) \cup \{v\not\en Ker(\rho^*-I)\colon \rho^*(v)-v\en R(\{X_n\})^{\perp}\}.\]
Here, $R(\{X_n\})$ is the subspace of $H$ in which $\{X_n\}$ takes its values.  Assuming that $R(\{X_n\})$ is a dense subspace of $H$, we have that $R(\{X_n\})^{\perp}=\{0\}$ a.s.  Therefore, in this scenario $W_X = U_X$.
\end{rem}

\begin{example}
Consider again the process in Example \ref{ex:K}.  The Unit Root space can be found by solving the system
\[V=A_K V.\]
The eigenvectors $V$ will provide the coefficients for the eigenfunctions
\[f ( s ) = \sum_{i} V_i a_i(s).\]
Also, the multiplicity of $1$ as an eigenvalue of $A_K$ will determine the dimension of $U_X$ and thus the number of common trends.
\end{example}

\begin{example}
In Example \ref{ex:B}, when $\alpha$ is the value producing non-stationarity, the eigenvector is easily seen to be
\[v = 0.8506 e_1 + 0.5257 e_2.\]
Therefore the space of common trends is one-dimensional.  Also, since $U_X\subseteq W_X$, we have another weak unit root.  Some calculations show that the space of all weak unit roots is characterized by
\[\prodint{X_n}{e_1}\prodint{v}{e_1+e_2}=-\prodint{X_n}{e_2}\prodint{v}{e_1}, \qqd\text{a.s.}\]
\end{example}

\begin{example}
Compact operators can be produced from arbitrary multiplicity functions $m(\lambda)$ by means of the spectral decomposition.  This means that given $k\en \mathbb{N}$, there exists a functional AR(1) process with unit root having exactly $k$ common trends.
Indeed, let $m(\lambda)$ be a multiplicity function defined on $\abs{\lambda}\le 1$ such that $m(1)=k$ and $m(\lambda_n)<1$ for all $\abs{\lambda_n}<1$. Since $H$ is separable, it admits a countable basis, say $\{e_i\}$. Define $H_1=sp(e_1,\dotsc, e_n)$, and for $j\ge 1$, let $H_j=sp(e_{j+n})$.  Denote by $P_j$ the projection onto $H_j$ and define
\[\rho=\sum_{j=1}^\infty \lambda_j P_j.\]
Theorem \ref{t:unitroots} shows that the process \eqref{e:far1} defined with $\rho$ has the desired property.
\end{example}

\section{Illustrations with demographic data}\label{s:rex}

In this Section two data sets are briefly analyzed for the presence of functional unit roots.  We use Functional Principal Component Analysis (FPCA) and base our conclusions on the observation that the coefficient processes for some of these components are non-stationary.  The form of FPCA that we use is the more standard one, as can be found in \cite{ramsay05} and is implemented with the statistical software of \cite{R.lang}.  Specifically, we use the function \texttt{ftsm} of the R package \texttt{ftsa}.

The reason we do not use the recent Dynamic form of FPCA given in \cite{hKh15} is that what we need for the study of functional unit roots is a decomposition of the space $H$ in terms of the variability of $\{X_n\}$.  We are not, as such, interested in reducing dimensionality and following \cite{hormann2010}, we know that the consistent estimation of functional principal components is possible under quite a general dependence framework.

The idea of using FPCA to detect functional unit roots follows that of \cite{Snell99} for multivariate time series. Intuitively, FPCA will first find the eigenfunctions corresponding to non-stationary projections since these carry most of the variance.  The coefficient process for each principal component in a non-stationary subspace will thus exhibit a unit root behavior.

\subsection{Male log-mortality in Italy}
Our first data sets consists of measurements of the male mortality in Italy.  The data set is available in the Human Mortality Database website (www.mortality.org).  The time span for the data ranges from 1872 to 2012.  Ages from $0$ to $110$ are included in each year.  We consider only male mortality in the age range of $10$ to $65$ avoiding some noisy measurements.  We focus on the years from 1959 to 2012 to be certain that we avoid the effect of the World Wars which, evidently, create higher mortality rates for distinct groups of age.

\figuraqui{Male log-mortality in Italy from 1959 to 2012}{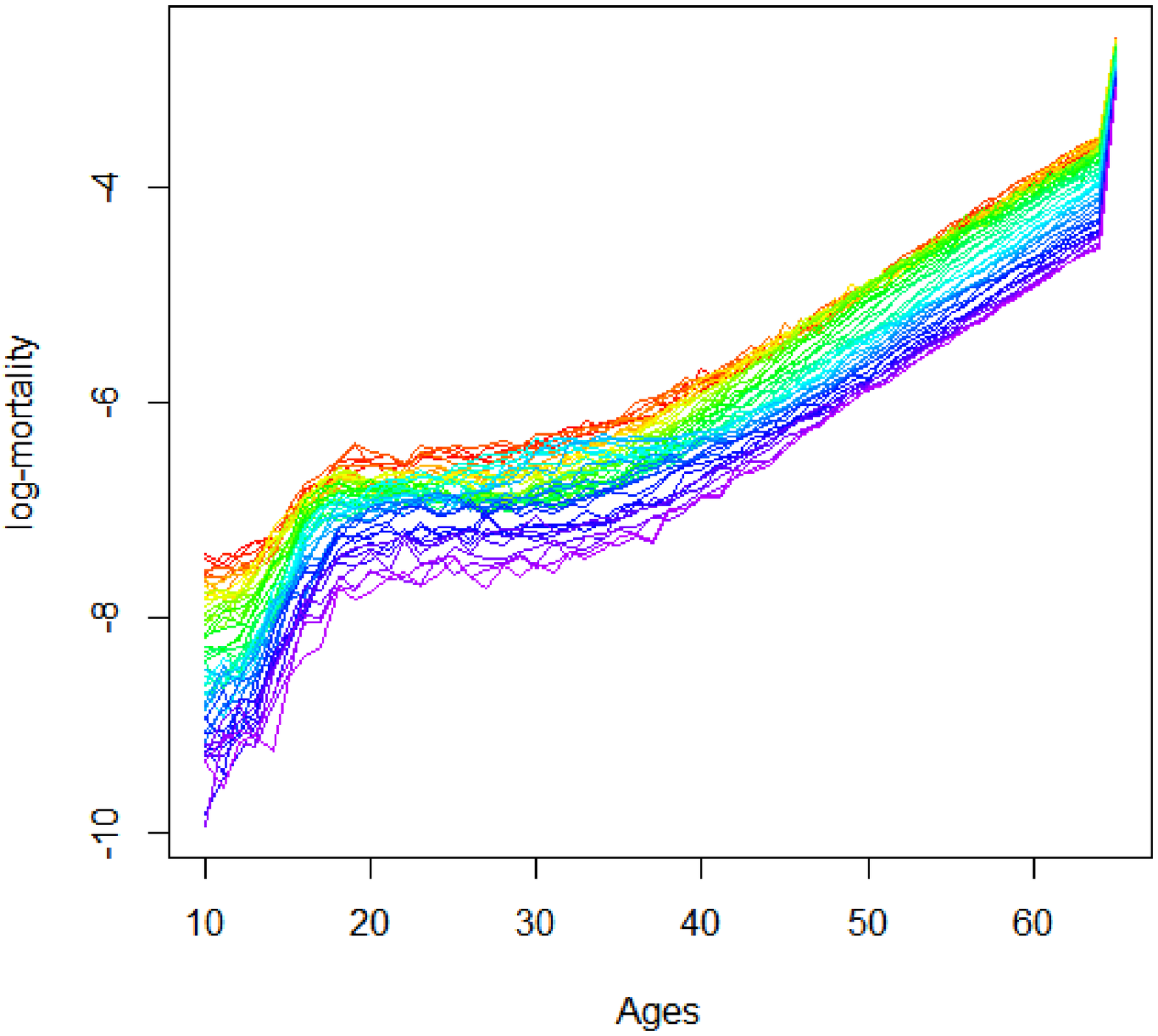}{3}\label{f:lm}

Figure 1 suggests that this functional time series is non-stationary since a clear decreasing rate in mortality is visible. What is not immediate is that this rate is not purely deterministic.

The coefficient process for each of the first three principal components is shown in Figure 2. An augmented Dickey--Fuller test on the first coefficient process indicates that it is best modeled by a simple random walk.  The test statistic is found to be $\tau=0.9599$ with critical values $-2.6, -1.95$ and $-1.61$ for the usual levels of significance.  This shows that the apparent decline in mortality contains both, a deterministic and a stochastic trend.

\figuraqui{Coefficients of the first three FPC for the Italian male log-mortality data. The continuous line corresponds to the first component}{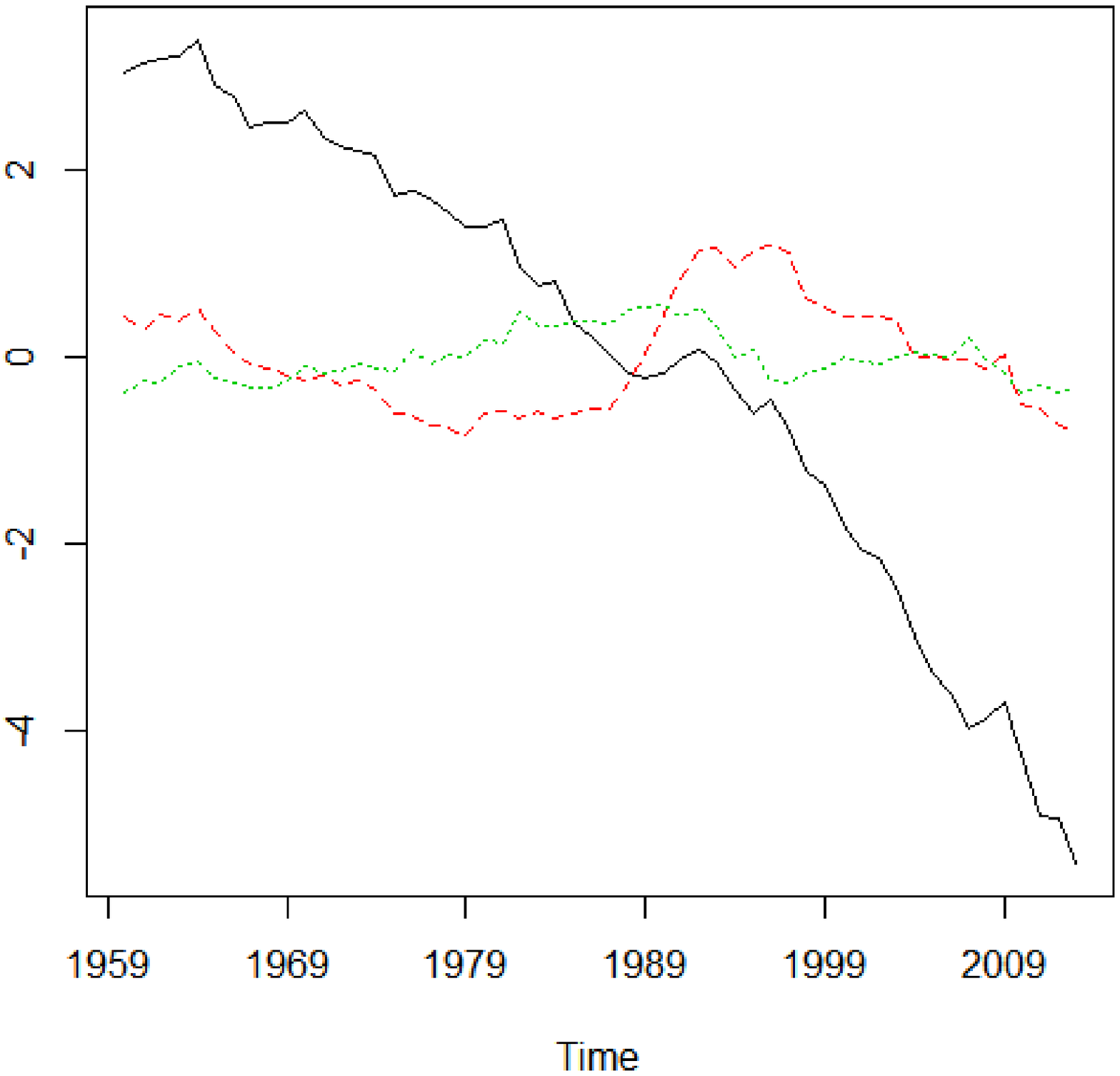}{3}

\subsection{Fertility rates in Australia}
The second data set we will use consists of age-specific fertility rates between ages 15 and 49 in Australia.  The time span of this data goes from 1921 to 2006.  It is accessible as a part of the R package \texttt{rainbow} and is smoothed with B-Splines.

\figuraqui{Australian Fertility Rates in Australia from 1921 to 2006.}{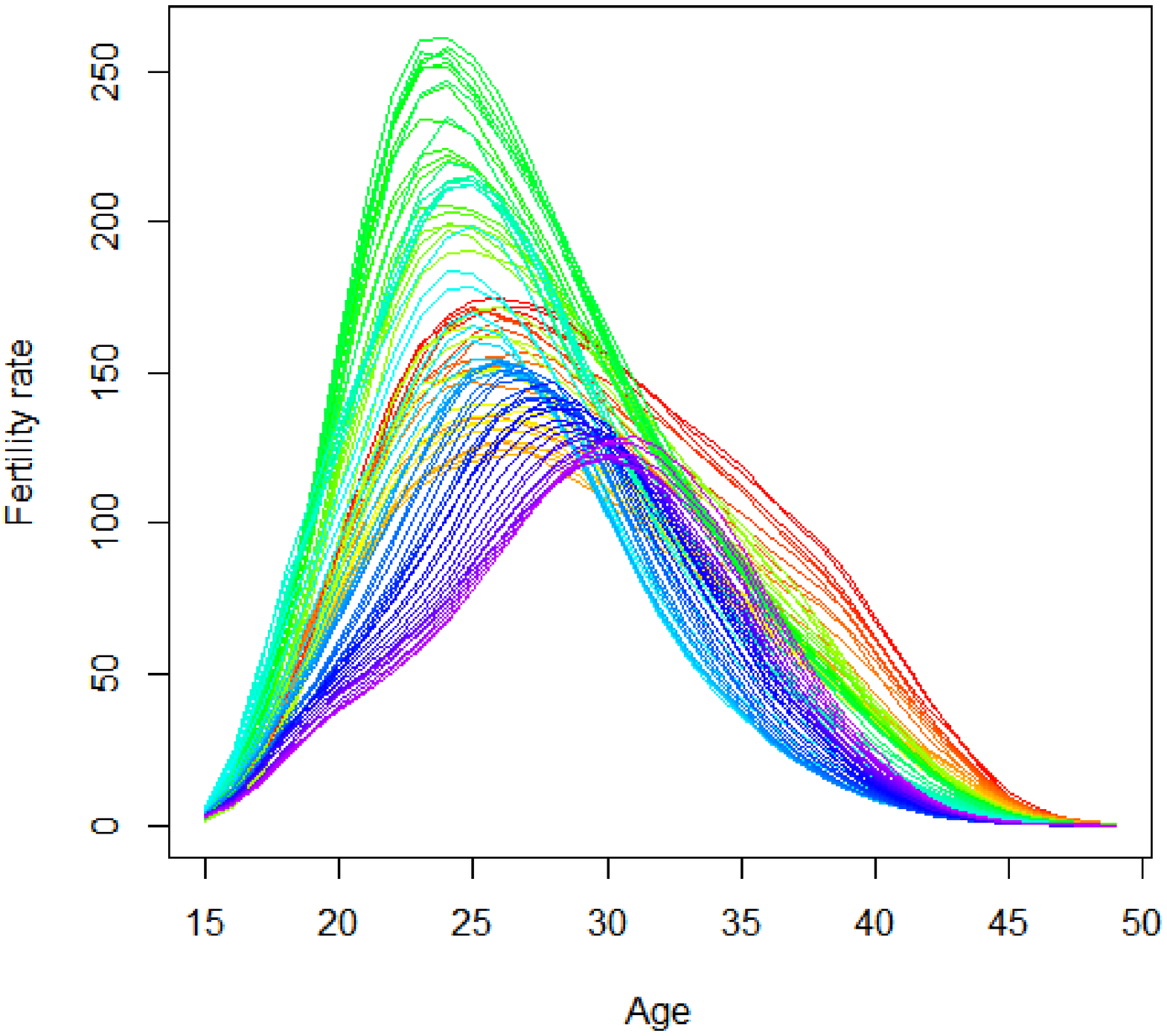}{3}

Figure 3 suggests non-stationarity and, as in the previous example, a decay in time.  Again, it is not clear if this non-stationarity is removable by some form of detrending.  As can be seen in Figure 4, the first two principal components induce apparently non-stationary coefficient processes.  Johansen's cointegration test was performed and at a significance level of $1\%$, the hypothesis of $0$ cointegrating relations cannot be rejected.   This result is consistent with the interpretation that the two first principal components are an empirical orthonormal basis of the common trends space.

\figuraqui{Coefficient processes for the first three principal components of the Australian Fertility Data. The dashed line corresponds to the last component.}{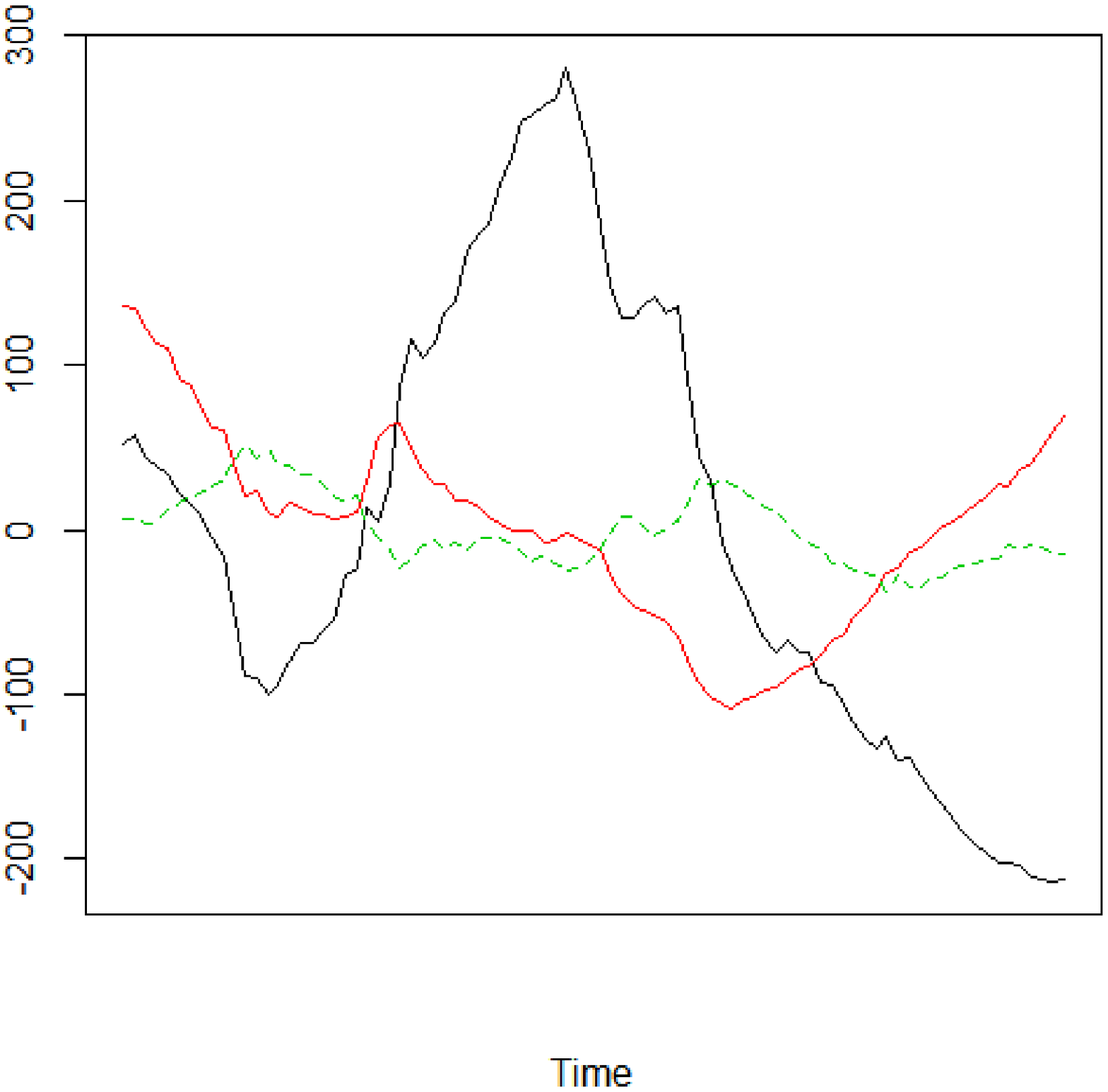}{3}

\section{Concluding Remarks}
To conclude our exposition, we make some remarks on our assumptions and on possible generalizations of our results.

\subsection{The order of the AR specification}
As we have tried to argue, the dynamics of the functional AR(1) are rich enough for most applications.  Following \cite{Kokoszka12}, however, the case $p=2$ is also worth examining. We will provide some remarks on this case.  Let $\rho_1, \rho_2\en\ops$ and define

\[X_n = \rho_1(X_{n-1})+\rho_2(X_{n-2})+\gep_n.\]

A common practice for the study of this process is using its Markovian representation in $H\otimes H$, namely
\[\left(\begin{array}{cc} X_{n}\\ X_{n-1}\end{array}\right) = \left(\begin{array}{cc} \rho_1&\rho2\\1&0 \end{array}\right) \left(\begin{array}{cc} X_{n-1}\\ X_{n-2}\end{array}\right)+\left(\begin{array}{cc} \gep_{n}\\ 0\end{array}\right).\]

The first difficulty that arises from this representation is that the operator thus constructed in $H\otimes H$ is not compact even if $\rho_1$ and $\rho_2$ both are, except when $H$ is finite-dimensional.  This is a consequence of the fact that the identity operator is only compact in such spaces.  Therefore, a treatment as the one given in this paper is not susceptible of generalization with this technique.  Nonetheless, something can be said.

Assume, for example, that both operators $\rho_1$ and $\rho_2$ are compact.  Let $\lambda_1$ and $\lambda_2$ be two respective eigenvalues and assume that $Ker(\rho_1^*-\lambda_1)\cap Ker(\rho_2^*-\lambda_2)\neq \emptyset$.   An element $v\en H$ thus exists such that
\[\prodint{X_n}{v}=\lambda_1\prodint{X_{n-1}}{v}+\lambda_2\prodint{X_{n-2}}{v}+\prodint{\gep_n}{v},\]
which shows that $v$ is a weak unit root if $\lambda_1+\lambda_2 = 1$.  Furthermore, the projections $\Pi_1$ and $\Pi_2$ onto $Ker(\rho-\lambda_1)$ and $Ker(\rho-\lambda_2)$ commute and the process $\{Y_n=\Pi_1\Pi_2 X_n\}$ can be represented as

\[\left(\begin{array}{cc} Y_{n}\\ Y_{n-1}\end{array}\right) = \left(\begin{array}{cc} \lambda_1&\lambda_2\\1&0 \end{array}\right) \left(\begin{array}{cc} Y_{n-1}\\ Y_{n-2}\end{array}\right)+\left(\begin{array}{cc} \gep_{n}\\ 0\end{array}\right).\]

Simple calculations show that $1$ is an eigenvalue of this new operator only when $\lambda_1+\lambda_2=1$.  We can then state that a unit root is present in an AR(2) process whenever
\begin{enumerate}
\item Both operators, $\rho_1$ and $\rho_2$ are compact,
\item $1\en\sigma(\rho_1)+\sigma(\rho_2)$,
\item There exist values $\lambda_i\en\sigma(\rho_i), i=1,2$, adding up to one such that $Ker(\rho_1^*-\lambda_1)\cap Ker(\rho_2^*-\lambda_2)\neq\emptyset.$
\end{enumerate}

Furthermore, if there are exactly $N$ pairs of null-spaces satisfying (3.) above, and if $\dim(Ker(\rho_1^*-\lambda_i)\cap Ker(\rho_2^*-\lambda_i))=n_i$, then the process has $n=\sum_{i=1}^N n_i$ common trends.

What this brief exam of the AR(2) specification intends to show is that despite the fact that a generalization of the techniques used in this paper is not immediate, some of the main ideas can indeed be useful.

From the point of view of applications, an empirical assessment of unit root behaviour is possible through the use of Functional Principal Component Analysis.

\subsection{Compacity of $\rho$}
Compacity played an important part in our developments.  The two most important consequences of this hypothesis are
\begin{enumerate}
\item The spectrum of $\rho$ is discrete, and every non-zero element of it is an eigenvalue
\item for every $\lambda\en\sigma(\rho)$, the dimension of $Ker(\rho-\lambda)$ is finite.
\end{enumerate}
It is the interplay of these structural properties of compact operators that makes functional unit roots similar to finite-dimensional cointegration.  As we have tried to argue before, this assumption is not quite a restrictive one since most operators used in empirical research are either finite-rank or integral, thus compact.

A first step in dropping the assumption of compacity is to assume that $\rho$ is \textit{power compact}, that is, $\rho^j$ is compact for some $j\en \mathbb{N}$.  In this case, for all $\lambda\neq 0$ the operator $\rho-\lambda$ is Fredholm and the spectrum of $\rho$ consists of eigenvalues of finite multiplicity.  Theorem \ref{t:unitroots} can then be proven by means of localization. See Proposition 1 of \cite{Konig2001941}.

Disregarding compacity altogether, we may define a unit root by requiring that $1$ is an eigenvalue of $\rho$.  This immediately gives
\[\prodint{X_n}{v}=\prodint{X_{n-1}}{v}+\prodint{\gep_n}{v},\]
for all $v\en Ker(\rho^*-1)$.  Therefore, a space of common trends exists and projection onto $Ker(\rho-1)$ is possible as in Theorem \ref{t:unitroots}.  The main difference is that this space may fail to be finite-dimensional.

\section*{Acknowledgments}
We are grateful for the  partial support form project  ECO2013-46395-P , Spain and
LEMME No. 264542, CONACyT , Mexico

\bibliographystyle{abbrvnat}
\setcitestyle{authoryear, open={(},close={)}}
\bibliography{biblioMaster}

\end{document}